\documentclass[12pt]{article}

\usepackage{comment}
\usepackage{pgf,tikz}
\usepackage{pgfplots}
\usetikzlibrary{arrows.meta}

\usepackage{amsmath,amssymb,amsthm,cite}
\usepackage[colorlinks,citecolor=red,urlcolor=blue,linkcolor=blue]{hyperref}
\usepackage{mathrsfs,dsfont}
\usepackage{graphicx}
\usepackage{enumitem}

\newtheorem{nummer}{ }
\newtheorem{thm}[nummer]{\bf Theorem}
\newtheorem{prop}[nummer]{\bf Proposition}
\newtheorem{lem}[nummer]{\bf Lemma}
\newtheorem{cor}[nummer]{\bf Corollary}

\newtheorem{ques}{\bf Question}

\newtheorem{rmk}{\bf Remark}
\newcommand{\rmksty}{\rm}
\newtheorem*{alg}{\bf Algorithm}
\newcommand{\algsty}{\rm}

\newcommand{\ie} {\sl i.e.}
\newcommand{\eg} {\sl e.g.}

\newcommand{\dpp}{double-pytha\-potent}
\newcommand{\qpp}{quadratic pytha\-potent}
\newcommand{\hppp}{pytha\-potent pair of degree~$h$}

\newcommand{\cppp}{cubic pytha\-potent pair}

\newcommand{\dppp}{double-pytha\-potent pair}

\newcommand{\qppp}{quadratic pytha\-potent pair}
\newcommand{\qppps}{quadratic pytha\-potent pairs}
\newcommand{\pp}{pytha\-gorean pair}
\newcommand{\pps}{pytha\-gorean pairs}

\newcommand{\tor}[2]{\Z/#1\Z\times\Z/#2\Z}

\newcommand{\Q}{\mathds{Q}}

\newcommand{\Z}{\mathds{Z}}

\newcommand{\THM}{Theorem}

\newcommand{\LEM}{Lemma}

\newcommand{\tq}{\tilde q}

\newcommand{\tp}{\tilde p}
\newcommand{\tb}{\tilde b}
\newcommand{\ta}{\tilde a}

\newcommand{\C}{\Gamma}

\newcommand{\Cab}{\C_{a,b}}
\newcommand{\Caabb}{\C_{a^2,b^2}}

\newcommand{\Cabh}{\C_{a^h,b^h}}

\makeatletter
\def\opargproof[#1]{\par\noindent {\bf #1 }}
 \makeatother

\definecolor{darkgreen}{rgb}{0,.6,0}

\begin{document}
\begin{center}
{\LARGE\bf Pairing Powers of Pytha\-gorean Pairs}

\medskip

{\small Lorenz Halbeisen}\\[1.2ex] 
{\scriptsize Department of Mathematics, ETH Zentrum,
R\"amistrasse\;101, 8092 Z\"urich, Switzerland\\ lorenz.halbeisen@math.ethz.ch}\\[1.8ex]
{\small Norbert Hungerb\"uhler}\\[1.2ex] 
{\scriptsize Department of Mathematics, ETH Zentrum,
R\"amistrasse\;101, 8092 Z\"urich, Switzerland\\ norbert.hungerbuehler@math.ethz.ch}\\[1.8ex]
{\small Arman Shamsi Zargar}\\[1.2ex] 
{\scriptsize Department of Mathematics and Applications, University of Mohaghegh Ardabili,
Ardabil, Iran\\ zargar@uma.ac.ir}
\end{center}

\hspace{5ex}{\small{\it key-words\/}: pytha\-gorean pair, pytha\-potent pair of arbitrary degree, elliptic curve}

\hspace{5ex}{\small{\it 2020 Mathematics Subject 
Classification\/}: {\bf 11D72} 11G05

\begin{abstract}\noindent
{\small A pair $(a,b)$ of positive integers is a {\it \pp\/} if 
\hbox{$a^2+b^2$} is a square. A pytha\-gorean pair $(a,b)$ is called {a {\it\hppp\/}} if there is another {\pp} $(k,l)$, which is not a multiple
of $(a,b)$, such that $(a^hk,b^hl)$ is a {\pp}. To each {\pp} $(a,b)$ we assign an elliptic curve $\Gamma_{a^h,b^h}$ for $h\geq 3$
with torsion group isomorphic to $\tor 2 4$ such that $\Gamma_{a^h,b^h}$ has positive rank over $\Q$ if and only if 
$(a,b)$ is {a \hppp}. As a side result, we get that if $(a,b)$ is {a {\hppp}}, then there exist infinitely many {\pps} $(k,l)$, not multiples of each other, such that
$(a^hk,b^hl)$ is a {\pp}. In particular, we show that any {\pp\/} is always a pytha\-potent pair of degree~$3$. In a previous work, pytha\-potent pairs of degrees~$1$ and $2$ have been studied.} 
\end{abstract}

\section{Introduction}
A {\it \pp\/} is a pair $(a,b)$ of positive integers such that 
\hbox{$a^2+b^2$} is a square. We adopt the usual notation
$a^2+b^2=\Box\,$ for this. A pytha\-gorean pair $(a,b)$ is called
a {\it\hppp\/} if there is another {\pp} $(k,l)$, which is not a multiple
of $(a,b)$, such that $(a^hk,b^hl)$ is also a {\pp}, {\ie},
$$a^2+b^2=\Box\,,\qquad k^2+l^2=\Box\,,\qquad\text{and}\qquad (a^h k)^2+(b^h l)^2=\Box\,.$$
To simplify the language we will call a pytha\-potent pair of degree~$3$, $4$ and $5$ a {\em cubic, quartic\/} and {\em quintic pythapotent pair}, respectively. We will also keep the definition of a {\dpp} and {\qppp} given in \cite{JNT} that address a pytha\-potent pair of degree~$1$ and $2$, respectively. 

As the pair of squares $(a^2,b^2)$ of a {\pp} $(a,b)$ is never a {\pp}, it is natural to ask whether the Hadamard--Schur products $(a^hk,b^h\ell), h \geq 1$, of {\pps} can be a {\pp} or not. For a {\dpp} and {\qppp}, the question has been investigated in \cite{JNT}. 
More precisely, it has been shown that for each {\pp} $(a,b)$, the elliptic curve $\Cab$ ($\Caabb$, resp.) has torsion group isomorphic to $\tor 2 4$ ($\tor 2 8$, resp.) and that 
$(a,b)$ is a {\dpp} ({\qpp}, resp.) pair if and only if $\Cab$ ($\Caabb$, resp.) has positive rank over $\Q$. With the points of infinite order on the curve $\Cab$ ($\Caabb$, resp.) infinitely many {\pps} $(k,l)$ can be generated, not multiples of each other, such that $(ak,bl)$ ($(a^2k,b^2l)$, resp.) are {\pps}. Moreover, every elliptic curve $\Gamma$
with torsion group $\tor 2 8$ is isomorphic to a curve of the form $\Caabb$ for some
{\pp} $(a,b)$. 

In this work, we will answer affirmatively the question for $h\geq 3$, which generalizes the results of \cite{JNT} concerning {\dpp} and {\qppps}. The question will lead, indeed, again in a natural way to associated elliptic curves of positive rank over $\Q$.

For a positive integer $h$ we assign the elliptic curve 
$$
\Cabh:\hspace*{3ex} y^2\ =\ x^3+(a^{2h}+b^{2h})x^2+a^{2h}b^{2h}x
$$
to the {\pp} $(a,b)$.
We will show that the curve $\Cabh$ with $h\geq 3$ has torsion group isomorphic to $\tor 2 4$ and that $(a,b)$ is a {\hppp} if and only if $\Cabh$ has positive rank over $\Q$.
With the points of infinite order on the curve $\Cabh$ we can generate
infinitely many {\pps} $(k,l)$, not multiples of each other, such that
$(a^h k,b^h l)$ are {\pps}.

\noindent {\bf Examples.} We give an example of {\cppp} and of a quartic  pytha\-potent pair. 
\begin{enumerate}
	\item 
	The {\pp} $(a,b)=(3,4)$ is a {\cppp}. Indeed, for the {\pp}  $(k,l)=(8,15)$ we have  
	$$(3^3\cdot 8)^2+(4^3\cdot 15)^2=984^2.$$
	Since the rank of the elliptic curve $\C_{3^4,4^4}$ is~$1$, $(3,4)$ is also a quartic pytha\-potent pair. 
    \item  
    	The {\pp} $(a,b)=(3,4)$ is a quartic pytha\-potent pair. Indeed, for the {\pp}  $(k,l)=(176,57)$ we have  
	$$(3^4\cdot 176)^2+(4^4\cdot 57)^2=20400^2.$$
	However, since the rank of the elliptic curve
    	$\C_{3^5,4^5}$ is~$0$, $(3,4)$ is not a quintic pytha\-potent pair.
   \end{enumerate}

\begin{rmk}
{\rmksty In \cite{JNT}, the parametrization $\Caabb$ 
for elliptic curves with torsion group $\tor 2 8$, 
where $(a,b)$ is a {\pp}, was obtained 
by using Schroeter's construction of cubic curves with 
line involutions (see \cite{schroeter_paper}).
Other new parametrizations obtained by Schroeter's 
construction for elliptic curves 
with torsion groups $\Z/10\Z$, $\Z/12\Z$, and $\Z/14\Z$ can be found in \cite{rabarison_paper}.
Furthermore, the curves $\C_{a^h,b^h}$, where $(a,b)$ is a {\pp},
are obtained by replacing the terms
$a^4+b^4$ and $a^4b^4$ in the parametrization of
$\Caabb$ by $a^{2h}+b^{2h}$ and $a^{2h}b^{2h}$, respectively.}
\end{rmk}

\section{Pytha\-potent Pairs of Degree~$h$}

We first show that the curve $\C_{a^h,b^h}$ has torsion group isomorphic to $\tor 2 4$ for any positive integer $h\neq 2$, and then we show how we obtain {\pps} $(k,l)$ from a point on $\C_{a^h,b^h}$ whose $x$-coordinate is a square such that $(a^hk,b^hl)$ is a {\pp}.

\begin{prop} If $(a,b)$ is a {\pp}, then the elliptic curve 
	$$
	\C_{a^h, b^h}:\hspace*{3ex} y^2\ =\ x^3+(a^{2h}+b^{2h})x^2+a^{2h}b^{2h}x\,,
	$$
	has torsion group $\tor 2 4$ for any positive integer $h\neq 2$.
\end{prop}

\begin{proof}
We use Kubert's parametrization for elliptic curves with torsion group $\tor 2 4$ 
(see~\cite[p.\,217]{Kubert}): 
$$y^2+xy-ey\ =\ x^3-ex^2$$
for $$e=v^2-\tfrac 1{16}\qquad\text{where $v\neq 0,\;\pm\tfrac 14$}\,.$$
By means of a rational transformation, this parameterization can be brought into the form 
	$$y^2\ =\ x^3+\ta x^2+\tb x$$
	with $$\ta=2\cdot(16v^2+1)\quad
	\text{and}\quad\tb=(16v^2-1)^2.$$
	Now, for $v=\frac pq$, let $p:=a^h-b^h$ and $q:=\frac 14(a^h+b^h)$.
	Then the curve $y^2\ =\ x^3+\ta x^2+\tb x$ is equivalent to the curve 
	$$\Cabh:\hspace{3ex}
	y^2\ =\ x^3+(a^{2h} +  b^{2h})x^2+a^{2h}b^{2h}x.$$
	This shows that for any positive integer $h$, the torsion group of $\Cabh$ is $\tor 2 4$ or $\tor 2 8$. But, thanks to the proof of \cite[Proposition~1]{JNT}, for a given {\pp} $(a,b)$, the torsion subgroup of $\Cabh$ is isomorphic to $\tor 2 8$ only when $h=2$. This completes the proof.  
\end{proof}

\begin{thm}\label{thm:main28} 
The {\pp} $(a,b)$ is a {\hppp} if and only if the elliptic curve $\Cabh$ 
has positive rank over $\Q$.
\end{thm}

The following Lemmas~\ref{lem:trans28}, \ref{lem:main24} and~\ref{lem:final24}
prepare the proof of {\THM}\;\ref{thm:main28}. First,  we transform the curve
$\Cabh$ to another curve on which we carry out our calculations. 

\begin{lem}\label{lem:trans28} 
If $(\bar x,\bar y)$, $\bar x\neq0$, is a point on the curve $\Cabh$, then the point $(\frac{a^hb^h}{\bar x},\frac{\bar y}{\bar x})$ is a point on
the curve $$y^2x\ =\ a^hb^h+(a^{2h}+b^{2h})x+a^hb^h x^2.$$ 
In particular, if $(\bar x,\bar y)$ is a rational point, then so is $(\frac{a^hb^h}{\bar x},\frac{\bar y}{\bar x})$.
\end{lem}

\begin{proof}
If $(\bar x,\bar y)$ lies on $\Cabh$, then $(\bar X,\bar Y,\bar Z)=(a^hb^h, \bar y,\bar x)$ is a point 
on the projective curve
$$
XY^2 = a^hb^h Z^3+(a^{2h}+b^{2h})XZ^2 + a^hb^h X^2Z.
$$
Since $\bar x=\bar Z\neq 0$, the point $(\bar X,\bar Y,\bar Z)$ also satisfies the dehomogenized equation
$$
\frac{X}{Z}\Bigl(\frac{Y}{Z}\Bigr)^2 = a^hb^h +(a^{2h}+b^{2h})\frac XZ +a^hb^h\Bigl(\frac{X}{Z}\Bigr)^2.
$$
Hence, $(\frac{\bar X}{\bar Z}, \frac{\bar Y}{\bar Z}) = (\frac{a^hb^h}{\bar x}, \frac{\bar y}{\bar x})$ is a point on the affine curve
$$y^2x\ =\ a^hb^h+(a^{2h}+b^{2h})x+a^hb^h x^2,$$
as claimed.
\end{proof}

Let $x_0=\frac{p}{q}$ be the $x$-coordinate of a rational point on
$y^2x=a^hb^h+(a^{2h}+b^{2h})x+a^hb^h x^2$, where $q=\tq^2$ and 
$p=a^hb^h\cdot\tp^2$ for some integers $\tq,\tp$. Then $$a^hb^h\cdot y^2\cdot\frac pq\;=\;
a^hb^h\cdot y^2\cdot\frac{a^hb^h\cdot \tp^2}{\tq^2}\;=\;
y^2\cdot\biggl(\frac{a^hb^h\cdot\tp}{\tq}\biggr)^2\;=\;\Box\,.$$
Therefore, $$a^hb^h\cdot\biggl(a^hb^h+(a^{2h}+b^{2h})\cdot\frac pq+a^hb^h\cdot\frac{p^2}{q^2}
\biggr)\;=\;\Box\,,$$
and by clearing square denominators we obtain
$$a^hb^h\cdot\bigl(a^hq+b^hp\bigr)\cdot\bigl(a^hp+b^hq\bigr)\;=\;\Box\,,$$
which is surely the case when
\begin{equation}
	a^h\cdot(a^hq+b^hp)\;=\;\Box\qquad\text{and}\qquad
	b^h\cdot(a^hp+b^hq)\;=\;\Box\,.\label{eq:24}  
\end{equation}
We will now use the group structure on elliptic curves to add points. 
In particular we will write $[2]P$ to denote the point $P+P$.
\begin{lem}\label{lem:main24}
	Let $P=(x_1,y_1)$ be a rational point on $\Cabh$ and let~$x_2$ be
	the $x$-coordinate of the point $[2]P$. 
	Then, for $x_0:=\frac{a^hb^h}{x_2}=\frac{p}{q}$ we have
	$q=\tq^2$ and $p=a^hb^h\cdot\tp^2$ for some integers $\tq,\tp$, 
	and $p$ and $q$ satisfy~\eqref{eq:24}.
\end{lem}

\begin{proof}
	By Silverman and Tate~\cite[p.~27]{SilvermanTate} we have
	$$x_2=\frac{(x_1^2-B)^2}{(2y_1)^2}\qquad\text{for }B:=a^{2h}b^{2h},$$
	and therefore
	$$x_0\;=\;\frac{a^hb^h}{x_2}\;=\;
	\frac{a^hb^h\bigl(4x_1^3+4 Ax_1^2+4 B x_1\bigr)}{(x_1^2-B)^2}\;=\;
	\frac{p}{q}\qquad\text{for }A:=a^{2h}+b^{2h}.$$
	So, $q=\Box$ and $p=a^hb^h\cdot\tp^2$ for some integer $\tp$.
	
	Now, for $p$ and $q$ (with $a=m^2-n^2$ and $b=2mn$) we obtain
	$$
	a^h\cdot(a^hq+b^hp)\;=\; a^{2h}\bigl(x_1^2+2b^{2h}+B\bigr)^2\;=\;\Box
	$$
	and
	$$
	b^h\cdot(a^hp+b^hq)\;=\; b^{2h}\bigl(x_1^2+2a^{2h}+B\bigr)^2\;=\;\Box
	$$
	which completes the proof.
\end{proof}

The next result gives a relation between rational points on $\Cabh$ 
with square $x$-coordinates and {\pps} $(k,l)$ such that $(a^h k,b^h l)$ is a {\pp}.

\begin{lem}\label{lem:final24} Let $(a,b)$ be a {\pp}. 
Then every {\pp} $(k,l)$ such that $(a^h k,b^h l)$  
	is a {\pp} corresponds to a rational point on $\Cabh$ 
	whose $x$-coordinate is a square,
	and vice versa.
\end{lem}

\begin{proof}
	Let $x_2=g^2/f^2$ be the $x$-coordinate of a rational point on
	$\Cabh$ for some rational values $f$ and $g$. 
	Then, by {\LEM}\;\ref{lem:main24}, 
	$\frac{a^hb^h}{x_2}=\frac{a^hb^h\cdot f^2}{g^2}$, 
	where $p=a^hb^h\cdot f^2$ and $q=g^2$ satisfy~\eqref{eq:24}, {\ie},
	$a^{2h}g^2+a^{2h}b^{2h}f^2=\Box$ and $b^{2h}g^2+a^{2h}b^{2h}f^2=\Box$. So, $\bigl(\frac{g}{f}\bigr)^2+
	b^{2h}=\rho^2$ for some $\rho\in\Q$ and $\bigl(\frac{g}{f}\bigr)^2+
	a^{2h}=\Box$. Let $\frac gf=\frac{2\rho t}{t^2+1}$ and 
	$b^h=\frac{\rho (t^2-1)}{t^2+1}$. Then $\rho=\frac{b^h(t^2+1)}{t^2-1}$ and
	$\frac gf=\frac{2 b^h t}{t^{2}-1}$, which gives us
	$$t=\frac{b^hf\pm\sqrt{g^2+b^{2h}f^2}}{g}\,.$$
	Since $$g^2+b^{2h}f^2\;=\;q+\frac{b^{2h}p}{a^hb^h}\;=\;q+\frac{b^hp}{a^h},$$ 
	by multiplying by $a^{2h}$ we get $$a^{2h}\cdot(g^2+b^{2h}f^2)\;=\;
	a^{2h}\cdot q+a^hb^h\cdot p\;=\;a^h(a^hq+b^hp).$$ Hence, by {\LEM}\;\ref{lem:main24},
	$g^2+b^{2h}f^2=\Box$ and therefore $t$ is rational, say $t=\frac rs$.
	Finally, since $\bigl(\frac{g}{f}\bigr)^2+a^{2h}=\Box$, we obtain
	$$a^{2h}\cdot(r^2-s^2)^2+b^{2h}\cdot(2rs)^2\;=\;\Box,$$ and for $k:=r^2-s^2$, $l:=2rs$, we finally
	get $$(a^{h} k)^2+(b^{h} l)^2\;=\;\Box\qquad\text{where $k^2+l^2=\Box$},$$
	which shows that $(a,b)$ is a {\hppp}.
	
	Assume now that we find a {\pp} $(k,l)$ such that $(a^h k,b^h l)$ is a {\pp}. 
	Without loss of generality we may assume that $k$ and $l$ are relatively prime.
	Thus, we find relatively prime positive integers $r$ and $s$ such that 
	$k=r^2-s^2$ and $l=2rs$. With $r,s,a,b$ we can compute 
	$x_2=\frac{b^{2h}l^2}{k^2}$, which is the $x$-coordinate of
	a rational point on $\Cabh$ whose $x$-coordinate 
	is obviously a square. 
\end{proof}

We are now ready for the proof of the main theorem.

\begin{proof}[Proof of Theorem\;\ref{thm:main28}]
	For every rational point $P$ on $\Cab$ with square $x$-coordinate
	let $(k_P,l_P)$ be the corresponding {\pp}. 
	By {\LEM}\;\ref{lem:final24} it is enough to show that no rational point
	with square $x$-coordinate has finite order.
	Notice that if $P$ is a point of infinite order, then for every
	integer $i$, $[2i]P$ is a rational point on $\Cabh$ with
	square $x$-coordinate, and not all of the corresponding {\pps}
	$(k_{[2i]P},l_{[2i]P})$ can be multiples of $(a,b)$.
	
	Let us consider the $x$-coordinates of the torsion points on 
	the curve $\Cabh$ with $h\geq 3$. For simplicity, we consider the $8$
	torsion points on the equivalent curve 
	$$y^2=\frac{a^hb^h}{x}+(a^{2h}+b^{2h})+a^hb^h x.$$
	The two torsion points at infinity are $(0,1,0)$ (which is the neutral
	element of the group) and $(1,0,0)$ (which is a point of order~$2$). 
	The other two points of order~$2$ are $(-\frac{a^h}{b^h},0)$ and
	$(-\frac{b^h}{a^h},0)$, and the four points of 
	order~$4$, which correspond to the four points 
	on the curve where the tangent to the curve
	is parallel to the $x$-axis, are
	$\bigl(1,\pm(a^h+b^h)\bigr)$ and $\bigl(-1,\pm(a^h-b^h)\bigr)$. 
	The corresponding points on the curve $\Cabh$
    are the three points 
    $(0,0)$, $(-b^{2h},0)$ and $(-a^{2h},0)$ of order~$2$,
    and the four points $\bigl(a^hb^h, \pm a^hb^h(a^h+b^h)\bigr)$ 
    and $\bigl(-a^hb^h, \pm a^hb^h(a^h-b^h)\bigr)$ of order~$4$.  
	Since $x_2$ is a square, we have that none of the values 
	$$ 0,\qquad -a^{2h},\qquad -b^{2h},\qquad
	a^{h}b^{h},\qquad -a^{h}b^{h},$$ is a rational square except 
	$0$ and possibly $a^hb^h$. 
	If $x_2=0$, then this implies $l=0$, but
	$(k,0)$ is not a {\pp}. 
	If $h$ is odd, $a^hb^h$ 
	cannot be a square unless $ab=\Box$, but it is impossible because   
    it is equivalent to the rank~$0$ congruent number 
    elliptic curve $y^2=x^3-x$ (also see~\cite[p.\;175]{Frenicle29}). 
    Consider now the case when $h$ is even. For the case
    $h=2$ see \cite{JNT}. If $h\geq 4$, 
    then by some algebra we obtain that $l^2=a^h b^{-h} k^2$, and therefore
    $(a^h k)^2+(b^h l)^2=a^h k^2(a^h + b^h)$. Thus, if 
    $(a^h k)^2+(b^h l)^2=\Box$, then also $$a^h + b^h=\Box\,.$$
    However, since $h\geq 4$, 
    by \cite[Main\;Theorem\,2.]{Darmon_Merel}
    this is impossible. 
    Thus, there is no {\pp} $(k,l)$ such that $(a^hk,b^hl)$ is a {\pp}.
\end{proof}

\begin{cor}
If $(a,b)$ is a {\hppp}, then there are 
infinitely many {\pps} $(k,l)$, not multiples of each other, such that
$(a^hk,b^hl)$ is a {\pp}.
\end{cor}

\begin{proof}
By {\THM}\;\ref{thm:main28}, there exists a point~$P$
on $\Cabh$ of infinite order. Now, for every
integer $i$, $[2i]P$ is a rational point on $\Cabh$ with
square $x$-coordinate, and each of the corresponding {\pps}
$(k_{[2i]P},l_{[2i]P})$ can be a multiple of just finitely many other
such {\pp}. Thus, there are infinitely many integers~$j$,
such that the {\pps} $(k_{[2j]P},l_{[2j]P})$ are not 
multiples of each other.
\end{proof}

\begin{alg}\label{alg1}{\algsty
		The following algorithm describes how to construct {\pps} $(k,l)$ from rational
		points on $\Cabh$ of infinite order.
		\begin{itemize}
			\item Let $P$ be a rational point on $\Cabh$ of infinite order and
			let $x_2$ be the $x$-coordinate of $[2]P$.
			\item Let $f$ and $g$ be relatively prime positive integers such
			that $$\frac{g}{f}\;=\;\sqrt{x_2}.$$
			\item Let $r$ and $s$ be relatively prime positive integers such
			that $$\frac rs\;=\;\frac{b^hf+\sqrt{g^2+b^{2h}f^2}}{g}.$$
			\item Let $k:=r^2-s^2$ and let $l:=2rs$.
		\end{itemize}
		Then $(a^hk,b^hl)$ is a {\pp}. 
	}
\end{alg}

\noindent {\bf Examples.} For $m=2$ and $n=1$, let $a=m^2-n^2$ and $b=2mn$. 
Then $(a,b)=(3,4)$ is a {\pp} and we have:
\begin{enumerate}
	\item For $h=1, 2, 5, 7, 10$, the rank of\/ $\Cabh$ is $0$. Hence, $(3,4)$ is not a {\hppp} for these $h$'s.
	\item The curve\/ $\C_{a^3,b^3}$, with torsion group $\tor 2 4$, has rank~$1$ with generator
	$P=(-3888, 50544)$. 
	The $x$-coordinate of $[2]P$ is $120^2$ which leads to 
	$(k,l)=(8, 15)$ with
	$$(3^3\cdot 8)^2+(4^3\cdot 15)^2={984}^2.$$
	\item The curve $\C_{a^4,b^4}$, with torsion group $\tor 2 4$, has rank~$1$ with generator
	$P=(-11616, 1779360)$. 
	The $x$-coordinate of $[2]P$ is $\left(\frac{912}{11}\right)^2$ which leads to 
	$(k,l)=(176, 57)$ with
	$$(3^4\cdot 176)^2+(4^4\cdot 57)^2={20400}^2.$$
	\item The curve $\C_{a^6,b^6}$, with torsion group $\tor 2 4$, has rank~$1$ with generator
	$P=\left(\frac{46022656}{9}, -\frac{678725632000}{27}\right)$. 
	The $x$-coordinate of $[2]P$ is $\left(\frac{3542528}{10335}\right)^2$ which leads to 
	$(k,l)=(82680, 6919)$ with
	$$(3^6\cdot 82680)^2+(4^6\cdot 6919)^2={66603976}^2.$$
	\item The curve $\C_{a^8,b^8}$, with torsion group $\tor 2 4$, has rank~$1$ with generator
	\begin{multline*}
	P=\bigl({7708125201644979550524801024}/{4449714580590446281}, \\  {1277921705702061766209671471345189388288000}\\
	/{9386382158955136069419053179}\bigr).
	\end{multline*}
	The $x$-coordinate of $[2]P$ is 
	\begin{multline*}
	$$\bigl({53440130127350994946668083276381874419712}\\
	/{5167588869543442260000066303720001735}\big)^2$$
	\end{multline*}
	which leads to 
	$$
	\begin{aligned}
	k&=165362843825390152320002121719040055520, \\
	l&=26093813538745603001302775037295837119
	\end{aligned}
	$$
    with
	\begin{multline*}
	(3^8\cdot 165362843825390152320002121719040055520)^2\\
		+(4^8\cdot 26093813538745603001302775037295837119)^2\\
	 ={2025214764653997025456624774452736238320416}^2.
	\end{multline*}
	\item  The curve $\C_{a^9,b^9}$, with torsion group $\tor 2 4$, has rank~$2$ with generators
	\begin{multline*}
	P=\bigl({25081364886334831007139600}/{1571551568609929201}, \\ {41443059164404768152156856423653413183280}\\
	/{1970121246732012270673366199}\bigr), 
	\end{multline*}
	and
	$$P'=\bigl({37016224137216}/{3481}, -{626321243401613475840}/{205379}\bigr).$$
 	The $x$-coordinate of $[2]P$ is
 	\begin{multline*}
 	\bigl({893146963381147972449638465746558838722398648}\\
 	/{1411411003643292904433293132813004363405}\bigr)^2
 	\end{multline*} 
 	which leads to 
	$$
	\begin{aligned}
		k&= 46249115767383421892470149376016526980055040, \\
		l&= 111643370422643496556204808218319854840299831
	\end{aligned}
	$$
	with
		\begin{multline*}
	(3^9\cdot 46249115767383421892470149376016526980055040)^2\\
	+(4^9\cdot 111643370422643496556204808218319854840299831)^2\\
	={29280793774283640248199258896077312831993114558464}^2,
    \end{multline*}
	and $x$-coordinate of $[2]P'$ is 
	$\left(\frac{33879841085325312}{2390157690995}\right)^2$ which leads to 
	$$(k',l')=(2390157690995, 129241337148)$$ with
	$$(3^9\cdot 2390157690995)^2+(4^9\cdot 129241337148)^2={57975169167761913}^2.$$
	Of course, we can also start with any other 
	rational point on $\C_{3^9,4^9}$, {\eg}, we can start with the point $Q=P+P'$. 
	The $x$-coordinate of $[2]Q$ is 
	\begin{multline*}
	\bigl({535606775034572770422692764010359062528}\\
	/
	{101246892970078905163938616171330325}\bigr)^2
	\end{multline*}
	which leads to 
	$$
	\begin{aligned}
	k&=809975143760631241311508929370642600, \\
	l&=16345421601396874097372215698558321
	\end{aligned}
	$$ with
	\begin{multline*}
	3^9\cdot 809975143760631241311508929370642600)^2\\
	+(4^9\cdot 16345421601396874097372215698558321)^2\\
	= 16508511692072764149750761383248338944776^2.
	\end{multline*}
\end{enumerate}

\begin{cor}
	Let $(a,b)$ is a {\pp} with $a=m^2-n^2$ and $b=2mn$ such that either $5m^2-n^2=\Box$ or $m^2+3mn+n^2=\Box\,$. Then $(a,b)$ is a {\dppp}.
\end{cor}
\begin{proof}
	One gets the quadratic conditions in the statement by imposing each of the points with $x$-coordinates 
	$n^2(m^2-n^2)$ and $mn(m-n)^2$ on the curve $\Gamma_{a^h, b^h}$ respectively. The result is now obtained from \cite[Algorithm~1]{JNT}.  
\end{proof}
\begin{cor}
	Let $(a,b)$ is a {\pp} with $a=m^2-n^2$ and $b=2mn$ such that 
	$$
	\begin{aligned}
	(i)\hspace{-.2cm} && -m^4-4mn^3+n^4  &=  \Box\,, & (iii)\hspace{-.2cm} && m^4-2m^3n+2m^2n^2+2mn^3+n^4 & = \Box\,, \\
	(ii)\hspace{-.2cm} &&  m^4+4m^2n^2-n^4  &=  \Box\,, & (iv)\hspace{-.2cm} && m^4-2m^3n-2m^2n^2-2mn^3+n^4 & = \Box\,. 
	\end{aligned}
	$$
	Then $(a,b)$ is a {\qppp}.
\end{cor}
\begin{proof}
	The quartic conditions (i)--(iv) are obtained by imposing each of the points with $x$-coordinates 
	$-8m^2n^4(m+n)^2$,
	$8m^4n^2(m^2-n^2)$, 
	$8m^3n^3(m^2-n^2)$,
	$-8m^3n^3(m+n)^2$,
	on the curve $\Gamma_{a^h, b^h}$ respectively.
	Note that each of the quartic conditions (i)--(iv) is equivalent to an elliptic curve of rank one. 
	The result now follows from \cite[Algorithm~2]{JNT}.
\end{proof}

\begin{cor}
	Let $(a,b)$ is a {\pp} with $a=m^2-n^2$ and $b=2mn$. Then $(a,b)$ is a {\cppp}. 
\end{cor}
\begin{proof}
	Since the curve $\C_{a^3,b^3}$ owns the non-obvious rational point 
	$$P=\bigl(-16(m^2-n^2)^2m^4n^4, 16(m^2-n^2)^2m^4n^4(m^2+n^2)(m^4-6m^2n^2+n^4)\bigr),$$
	the result immediately comes from Algorithm~\ref{alg1} as follows. The $x$-coordinate of the point $[2]P$ is 
	$$\bigl({2m^2n^2(m-n)^2(n+m)^2}/{(m^2+n^2)}\bigr)^2$$
	which, by applying Algorithm~\ref{alg1}, leads to
	$$k=4mn(m^2+n^2), \qquad l=(m-n)^2(n+m)^2,$$
	with 
	$$(a^3k)^2+(b^3l)^2=(4mn(m^4+n^4)(m-n)^2(n+m)^2)^2.$$
\end{proof}

We conclude the paper with two open problems.
	\begin{ques}
	Given an arbitrary {\pp} $(a,b)$. 
	Is there an $h\geq 4$ such that $(a,b)$ is a {\hppp}?
	\end{ques}
	
	\begin{ques}
	Given an arbitrary $h\geq 4$.
	Is there a {\pp} $(a,b)$ which is a {\hppp}? Or equivalently: Is there a {\pp} $(a,b)$
	such that $\Cabh$ has positive rank over $\Q$?
	\end{ques}

\noindent{\bf Declarations of interest:} none



\bibliographystyle{plain}

\end{document}